\documentclass[11pt, a4paper, reqno]{amsart}

\usepackage{amsxtra}
\usepackage{graphicx}
\usepackage{newlfont}
\usepackage{amscd}
\usepackage{hyperref}
\usepackage[german,english]{babel}
\usepackage{cancel}
\usepackage{amsmath,amsthm, mathtools}
\usepackage{amssymb}
\usepackage{enumerate}
\usepackage{color}
\usepackage[all,cmtip]{xy}
\usepackage{blindtext}
\usepackage{enumitem}

 \newtheorem{thm}{Theorem}[section]

 \newtheorem{lem}[thm]{Lemma}
 \newtheorem{prop}[thm]{Proposition}
 
  \theoremstyle{definition}

  \newtheorem{defn-thm}[thm]{Definition-Theorem}

 \theoremstyle{remark}

 \newtheorem{rem}[thm]{Remark}

\numberwithin{equation}{section}

\numberwithin{thm}{section}

\numberwithin{table}{section}

\numberwithin{figure}{section}

\ifx\pdfoutput\undefined
  \DeclareGraphicsExtensions{.pstex, .eps}
\else
  \ifx\pdfoutput\relax
    \DeclareGraphicsExtensions{.pstex, .eps}
  \else
    \ifnum\pdfoutput>0
      \DeclareGraphicsExtensions{.pdf}
    \else
      \DeclareGraphicsExtensions{.pstex, .eps}
    \fi
  \fi
\fi

\newcommand{\ep}{\varepsilon}

\newcommand{\R}{\mathbb{R}}

\begin{document}

\title{Uniqueness and symmetry of ground states for higher-order equations}
\author{Woocheol Choi}
\address{Department of Mathematics Education, Incheon National University, Incheon 22012, Korea}
\email{choiwc@inu.ac.kr}

\author{Younghun Hong}
\address{Department of Mathematics, Yonsei University, Seoul 03722, Korea}
\email{younghun.hong@yonsei.ac.kr}

\author{Jinmyoung Seok}
\address{Department of Mathematics, Kyonggi University, Suwon 16227, Korea}
\email{jmseok@kgu.ac.kr}

\thanks{}

\maketitle

\setcounter{tocdepth}{-1}

\pagestyle{plain}

\begin{abstract}
We establish uniqueness and radial symmetry of ground states for higher-order nonlinear Schr\"odinger and Hartree equations whose higher-order differentials have small coefficients. As an application, we obtain error estimates for higher-order approximations to the pseudo-relativistic ground state. Our proof adapts the strategy of Lenzmann \cite{L2} using local uniqueness near the limit of ground states in a variational problem. However, in order to bypass difficulties from lack of symmetrization tools for higher-order differential operators, we employ the contraction mapping argument in our earlier work \cite{CHS2} to construct radially symmetric real-valued solutions, as well as improving local uniqueness near the limit. 
\end{abstract}

\section{Introduction}

Higher-order elliptic equations, whose higher-order differentials have small coefficients, arise in various physical contexts. For instance, in nonlinear optics, the envelope dynamics of wave trains in a weakly nonlinear medium is given by the equation 
$$i\epsilon^2\partial_t\psi=\omega(\epsilon\partial) \psi-\epsilon^2|\psi|^2\psi,$$
where $\ep >0$, $\psi=\psi(t,x):\mathbb{R}\times\mathbb{R}^d\to\mathbb{C}$ and $\omega(\partial)$ denotes the Fourier multiplier operator with a symbol $a=a(\xi):\mathbb{R}^d\to\mathbb{R}$. Looking for a stationary solution, inserting the ansatz $\psi(t,x)=e^{i\mu t}u(x)$ with $\mu>0$, we obtain the time-independent equation
$$\omega(\epsilon\partial)u+\epsilon^2\mu u=\epsilon^2|u|^2u.$$
When high-frequency dispersion is negligible and the medium is isotropic\footnote{with $\omega(0)=\nabla_{\xi_j} \omega(0)=0$ and $\partial_{\xi_j}\partial_{\xi_k}\omega(0)=\delta_{jk}$ by a suitable change of variables.}, the above equation can be approximated by the second-order equation
$$-\Delta u+\mu u=|u|^2u$$
(see \cite{SS}). However, if high frequency dispersion is weak but not negligible,  one should consider a higher-order equation whose differential operator is a Taylor polynomial of $\omega(\epsilon\partial)$. Here, higher-order terms have small coefficients. 

In astrophysics, the mean-field limit of stationary boson stars is described by the pseudo-relativistic nonlinear Hartree equation
\begin{equation}\label{pNLH}
\left(\sqrt{-c^2\Delta+m^2c^4}-mc^2\right)u+\mu u=\left(|x|^{-1}*|u|^2\right)u,
\end{equation}
where $u=u(x):\mathbb{R}^3\to\mathbb{C}$, $m>0$ is the particle mass and $c>0$ stands for the speed of light. In applications, taking the formal Taylor polynomial of the pseudo-relativistic operator
\begin{equation}\label{pseudo-relativistic operator}
\left(\sqrt{-c^2\Delta+m^2c^4}-mc^2\right)=mc^2\left(\sqrt{1-\frac{\Delta}{m^2c^2}}-1\right)=\frac{1}{2m}(-\Delta)-\frac{1}{8m^3c^2}(-\Delta)^2+\cdots,
\end{equation}
the higher-order model
\begin{equation}\label{hpNLH}
\left(\sum_{j=1}^J\frac{(-1)^{j-1}\alpha_j}{m^{2j-1}c^{2j-2}}(-\Delta)^j\right)u+\mu u=\left(|x|^{-1}*|u|^2\right)u,
\end{equation}
where $\alpha_j=\frac{(2j-2)!}{j!(j-1)!2^{2j-1}}$, is employed to avoid possible complication from having a non-local operator (see \cite{CM, CLM} and the references therein).

Moreover, given a previously known second-order model 
$$-\Delta u+\mu u=f(|u|^2)u,$$
a higher-order equation is sometimes introduced as a refinement taking additional physical effects in account. In this case, it is natural to put small coefficients on higher-order differentials, like
$$-\Delta u+\epsilon\Delta^2u+\mu u=f(|u|^2)u,$$
for consistency with the second-order model.\\

The purpose of this paper is to provide a general strategy to prove uniqueness and radial symmetry of ground states for a certain class of higher-order elliptic equations including the above examples.

Before proceeding, it should be pointed out that proving uniqueness and symmetry of ground states for higher-order equations is in general quite challenging. That is because some of useful tools, such as the diamagnetic inequality, the P\'olya-Szeg\"o inequality, the moving plane method and the shooting game argument, might not be available. Recall that for second-order equations, the standard variational approach employs the diamagnetic inequality $\|\nabla(|u|)\|_{L^2}\leq \|\nabla u\|_{L^2}$ in the first step in order to obtain a non-negative ground state from a hypothetical possibly sign-changing ground state, and then symmetrization tools are applied to prove symmetry and uniqueness. When the symbol of a pseudo-differential operator is a Bernstein function, e.g., the pseudo-relativistic operator \eqref{pseudo-relativistic operator}, the diamagnetic inequality as well as symmetrization tools can be recovered by a beautiful argument in \cite{LY, Daubechies} involving the Bernstein's theorem. However, this method does not work for higher-order operators.

In fact, some of analytic tools have been developed for higher-order operators e.g., for polyharmonic operators, and there might be a way to apply them for uniqueness and symmetry. For a comprehensive overview, we refer to the book by Gazzola, Grunau and Sweers \cite{GGS}. Nevertheless, they cannot be directly applied to the above examples. Even worse, the desired diamagnetic inequality does not seem to hold for higher-order differential operators, because even if $u$ is smooth, its second derivative $\nabla_{x_j}^2(|u|)$ could be very singular near the set $\{x: u(x)=0\}$. In this paper, we go around the lack of the analytic tools rather than making an effort to build them up, by taking the advantage of higher-order differentials having small coefficients.\\

From now on, for concreteness of exposition, we restrict ourselves to the higher-order nonlinear Schr\"odinger equation (NLS)
\begin{equation}\label{hNLS}
P_\epsilon u+ u=|u|^{2k}u,
\end{equation}
where $k\in\mathbb{N}$ and $u=u(x):\mathbb{R}^d\to\mathbb{C}$, and the three-dimensional higher-order nonlinear Hartree equation (NLH)
\begin{equation}\label{hNLH}
P_\epsilon u+u=\left(|x|^{-1}*|u|^2\right)u,
\end{equation}
where $u=u(x):\mathbb{R}^3\to\mathbb{C}$. For an even integer $J$ and $\epsilon\geq0$ (including zero), the higher-order differential operator $P_\epsilon$ is defined by
$$P_\epsilon=P_\epsilon^J=:-\Delta+\sum_{|\alpha|=3}^Jc_\alpha\epsilon^{|\alpha|-2} (i\nabla)^\alpha,$$
where $\alpha=(\alpha_1,\cdots,\alpha_d)\in(\mathbb{Z}_{\geq 0})^d$ denotes a multi-index and $(i\nabla)^\alpha=i^{|\alpha|}\nabla_{x_1}^{\alpha_1}\cdots\nabla_{x_d}^{\alpha_d}$. We assume that the family of operators $\{P_{\epsilon}\}_{0\leq\epsilon\leq 1}$ is uniformly elliptic in the sense that there exists $\gamma>0$, independent of $\epsilon\geq0$, such that
\begin{equation}\label{uniform lower bound}
1+P_\epsilon\geq \gamma(1-\Delta).
\end{equation}
For NLS \eqref{hNLS}, we further assume that $1\leq d\leq 3$ and
\begin{equation}\label{subcritical}
\left\{\begin{aligned}
&k\in\mathbb{N}&&\textup{if }d=1,2,\\
&k=1&&\textup{if }d=3
\end{aligned}\right.
\end{equation}
so that the odd-power nonlinearity is $H^1$-subcritical. We remark that as $\epsilon\to 0$, the higher-order NLS \eqref{hNLS} formally converges to the standard second-order NLS
\begin{equation}\label{reference NLS}
-\Delta u+ u=|u|^{2k}u,
\end{equation}
while the higher-order NLH \eqref{hNLH} converges to the second-order NLH
\begin{equation}\label{reference NLH}
-\Delta u+ u=\left(|x|^{-1}*|u|^2\right)u.
\end{equation}

A solution to the higher-order NLS \eqref{hNLS} (resp., the higher-order NLH \eqref{hNLH}) is called a \textit{ground state} if it is a minimizer for the action functional
$$
I_\epsilon(u):=\left\{\begin{aligned}
&\frac{1}{2}\int_{\mathbb{R}^d}(P_\epsilon+1) u\bar{u}\,dx-\frac{1}{2k+2}\int_{\mathbb{R}^d}|u|^{2k+2}dx\quad\textup{(for NLS \eqref{hNLS})}\\
&\frac{1}{2}\int_{\mathbb{R}^3}(P_\epsilon+1) u\bar{u}\,dx-\frac{1}{4}\int_{\mathbb{R}^3}\left(|x|^{-1}*|u|^2\right)|u|^2dx\quad\textup{(for NLH \eqref{hNLH})}.
\end{aligned}\right.$$
restricted to the constraint
$$\langle I_\epsilon'(u),u\rangle_{L^2}=0\quad\textup{and}\quad u\neq0,$$
where $I_\epsilon'$ is the Frech\'et derivative of $I_\epsilon$. When $\epsilon=0$, it is known that the second-order NLS \eqref{reference NLS} (resp., the second-order NLH \eqref{reference NLH}) has a smooth radially symmetric positive ground state, denoted by $Q_0$, and it is unique up to translation and phase shift\footnote{We say that $Q$ is a unique solution up to translation and phase shift if for any solution $u$, there exist $x_0\in\mathbb{R}^d$ and $\theta\in\mathbb{R}$ such that $u(x)=e^{i\theta}Q(x-x_0)$.}. Moreover, the ground state $Q_0$ is non-degenerate. Indeed, linearizing the equation near the ground state $Q_0$, we obtain the linearized operator $\mathcal{L}=\begin{psmallmatrix}\mathcal{L}_0^+ &0\\0&\mathcal{L}_0^-\end{psmallmatrix}$ with the identification $a+bi\leftrightarrow \begin{psmallmatrix}a\\b\end{psmallmatrix}$, where the linear operators $\mathcal{L}_0^\pm: H^2(\mathbb{R}^d;\mathbb{R})\to L^2(\mathbb{R}^d;\mathbb{R})$ are defined by
$$\mathcal{L}_0^+h:=\left\{\begin{aligned}
&-\Delta h+h-(2k+1)Q_0^{2k}h&&\quad\textup{(for NLS \eqref{reference NLS})}\\
&-\Delta h+h-2\left(|x|^{-1}*(Q_0h)\right)Q_0-\left(|x|^{-1}*Q_0^2\right)h&&\quad\textup{(for NLH \eqref{reference NLH})}
\end{aligned}\right.$$
and 
$$\mathcal{L}_0^-h:=\left\{\begin{aligned}
&-\Delta h+h-Q_0^{2k}h&&\quad\textup{(for NLS \eqref{reference NLS})}\\
&-\Delta h+h-\left(|x|^{-1}*Q_0^2\right)h&&\quad\textup{(for NLH \eqref{reference NLH})},
\end{aligned}\right.$$
By \textit{non-degeneracy}, we mean that the kernels of $\mathcal{L}_0^\pm$ are explicitly given by
$$\left\{\begin{aligned}
\textup{Ker}\mathcal{L}_0^+ &={\textup{span}}\left\{\partial_{x_1}Q_0,\cdots, \partial_{x_d}Q_0\right\},\\
\textup{Ker}\mathcal{L}_0^- &=\textup{span}\left\{Q_0\right\}.
\end{aligned}\right.$$
When $\epsilon>0$, by standard variational arguments, one can show that the higher-order NLS \eqref{hNLS} (resp., the higher-order NLH \eqref{hNLH}) possesses a ground state $Q_\epsilon^J$ and that it converges to $Q_0$ as $\epsilon\to 0$ (see Proposition \ref{existence}). \\

Our main theorem establishes uniqueness and radial symmetry of ground states for the higher-order equations \eqref{hNLS} and $\eqref{hNLH}$.
\begin{thm}[Uniqueness and symmetry]\label{main thm}
Suppose that \eqref{uniform lower bound} (as well as \eqref{subcritical} for NLS) holds. Then, there exists $\epsilon_0>0$ such that for each $0<\epsilon\leq\epsilon_0$, there exists a smooth radially symmetric and real-valued ground state $Q_\epsilon^J$ for the higher-order NLS \eqref{hNLS} (resp., the higher-order NLH \eqref{hNLH}), and it is unique up to translation and phase shift. Moreover, the ground state $Q_\epsilon^J$ is non-degenerate in the sense of Proposition \ref{Non-degeneracy epsilon} below.
\end{thm}

For the proof, we follow the roadmap in the important work by Lenzmann \cite{L2}, where uniqueness of ground states for the pseudo-relativistic NLH \eqref{pNLH} is established. The robust approach of Lenzmann \cite{L2} can be summarized in two steps.
\begin{description}[align=left]
\item[Step 1] Construct a ground state for the pseudo-relativistic NLH \eqref{pNLH}, and prove its convergence to the ground state $Q_0$ for the second-order NLH \eqref{reference NLH} as $c\to \infty$ up to translation and phase shift. Here, by construction (involving variational techniques), a pseudo-relativistic ground state must be positive and radially symmetric.
\item[Step 2] Prove uniqueness of a radially symmetric real-valued solution to the pseudo-relativistic NLH near the ground state $Q_0$. The proof of this local uniqueness relies on the non-degeneracy of the ground state $Q_0$, which is indeed one of the main contributions of Lenzmann \cite{L2}.
\item[Conclusion] If $c\geq 1$ is large enough, then a pseudo-relativistic ground state is close to $Q_0$, so it is unique up to translation and phase shift.
\end{description}

As for higher-order equations, however, we cannot make use of radial symmetry of a ground state for the proof of uniqueness, but we have to prove it instead, since we do not have symmetrization tools at hand. In order to overcome these obstacles, we employ several new ingredients, including the contraction mapping argument  in our earlier work \cite{CHS2} and the improved local uniqueness near the ground state $Q_0$. Our proof can be summarized as follows.
\begin{description}[align=left]
\item[Step 1] Construct a ground state $Q_\epsilon^J$ for the higher-order equation \eqref{hNLS} (resp., \eqref{hNLH}), and prove its convergence to the ground state $Q_0$ for the second-order equation \eqref{reference NLS} (resp., \eqref{reference NLH}) as $\epsilon\to0$ up to translation and phase shift. We remark that contrary to Step 1 in \cite{L2}, due to lack of symmetrization tools, it is not known that $Q_\epsilon^J$ is radially symmetric and real-valued.
\item[Step 2] Construct a radially symmetric real-valued solution $u_\epsilon$ for the higher-order equation converging to the ground state $Q_0$ by the contraction mapping argument. By construction, the solution $u_\epsilon$ does not have any variational character as a ground state.
\item[Step 3] Prove uniqueness (up to translation and phase shift) for the higher-order equation near the ground state $Q_0$ without assuming that solutions are radially symmetric or real-valued. 
\item[Conclusion] If $\epsilon>0$ is small enough, then two solutions $Q_\epsilon^J$ and $u_\epsilon$ are close to the ground state $Q_0$. Thus, identifying them by uniqueness, we conclude that $Q_c^J$ is a unique radially symmetric real-valued ground state.
\end{description}
For the proof of local uniqueness in Step 2, we assume that there is a solution $\tilde{u}_\epsilon$, and then modify it by translation and phase shift to be perpendicular to the kernel of the linearized operator around $u_\epsilon$. Then, we prove that the modified $\tilde{u}_\epsilon$ is indeed $u_\epsilon$ itself. This argument, choosing the best modulation parameters, seems quite natural in the context of orbital stability \cite{W}. However, to the best of authors' knowledge, such a local uniqueness and its proof seem new. \\

Next, we consider the pseudo-relativistic NLH \eqref{pNLH} and the higher-order NLH \eqref{hpNLH} in the non-relativistic regime ($c\gg 1$). In this case, it is shown in Lenzmann \cite{L2} that the pseudo-relativistic NLH \eqref{pNLH} has a radially symmetric positive ground state $Q_c$, which is unique up to translation and phase shift. On the other hand, by the main theorem of this paper, the higher-order NLH \eqref{hpNLH} also has a radially symmetric and real-valued ground state $Q_c^J$, and it is unique up to translation and phase shift. Then, the contraction mapping argument in \cite{CHS2} can be applied to compare two ground states. As a consequence, we obtain the following error estimates for the higher-order approximations to the pseudo-relativistic ground state.

\begin{thm}[Higher-order approximations to a pseudo-relativistic ground state]\label{approximation}
Let $J$ be an odd number, and $c>0$ be a sufficiently large number. We denote by $Q_c$ (resp., $Q_c^J$) the unique radially symmetric, real-valued ground state for the pseudo-relativistic NLH \eqref{pNLH} (resp., the higher-order NLH \eqref{hpNLH}). Then,
$$\|Q_c^J-Q_c\|_{H^1}\lesssim\frac{1}{c^{2J}}.$$ 
\end{thm}

\begin{rem}
The higher-order Schr\"odinger operator in \eqref{hpNLH} is introduced as a higher-order approximation to the pseudo-relativistic operator $\sqrt{-c^2\Delta+m^2c^4}-mc^2$, provided that high-frequencies are not dominant. In \cite{CM, CLM}, the error estimates for the higher-order approximation to the linear evolution has been discussed. Theorem \ref{approximation} first provides a precise error estimate for the higher-order approximation to the pseudo-relativistic ground state, which the simplest nonlinear object.
\end{rem}

\subsection{Notations}
We denote the potential energy functional by
$$
\mathcal{N}(u):=\left\{\begin{aligned}
&\frac{1}{2k+2}\int_{\mathbb{R}^d}|u|^{2k+2}dx&&\textup{(for NLS)}\\
&\frac{1}{4}\int_{\mathbb{R}^3}\left(|x|^{-1}*|u|^2\right)|u|^2dx&&\textup{(for NLH)}.
\end{aligned}\right.$$
Then, the nonlinearity of the equation is given as its Frech\'et derivative 
$$\mathcal{N}'(u):=\left\{\begin{aligned}
&|u|^{2k}u&&\textup{(for NLS)}\\
&\left(|x|^{-1}*|u|^2\right)u&&\textup{(for NLH)}.
\end{aligned}\right.$$
We denote by $H_{P_\epsilon}^1=H_{P_\epsilon}^1(\mathbb{R}^3;\mathbb{C})$ the Hilbert space equipped with the inner product
$$\langle f,g\rangle_{H_{P_\epsilon}^1}:=\int_{\mathbb{R}^3} \left(P_\epsilon+1\right)f(x)\overline{g(x)}dx,$$

\subsection{Organization of the paper}
In Section 2, we prove existence of ground states $Q_\epsilon^J$'s for the higher-order NLS (resp., the higher-order NLH) and their convergence to the ground state $Q_0$ for the second-order equation. In Section 3, we provide the non-degeneracy estimates, which are the key analytic tools in this paper. Using them, in Section 4, we construct radially symmetric real-valued solutions $u_\epsilon$'s converging to the ground state $Q_0$ for the second-order equation. In Section 5, we establish local uniqueness for higher-order equations near the ground state $Q_0$, and then identifying $Q_\epsilon^J$ and $u_\epsilon$, we prove the main theorem (Theorem \ref{main thm}). Finally, in Section 6, we prove the error estimates for the higher-order approximation to the pseudo-relativistic ground state (Theorem \ref{approximation}).

\subsection{Acknowledgement} 
This research of the second author was supported by Basic Science Research Program through the National Research Foundation of Korea(NRF) funded by the Ministry of Education (NRF-2017R1C1B1008215). This research of the third author was supported by Basic Science Research Program through the National Research Foundation of Korea(NRF) funded by the Ministry of Education (NRF-2017R1D1A1A09000768).

\section{Construction of ground states, and their limit}

By the standard variational method, we construct ground states for higher-order equations (Proposition \ref{existence}), and show their convergence to the ground state for the second-order equation (Proposition \ref{convergence}). In addition, we prove that in general, convergence to the ground state for the second-order equation in a low regularity norm can be upgraded to that in high regularity norms (Proposition \ref{upgraded convergence}).

\begin{prop}[Existence of a ground state]\label{existence}
Suppose that \eqref{uniform lower bound} (as well as \eqref{subcritical} for NLS) holds. Then, for any $\epsilon>0$, the higher-order equation \eqref{hNLS} (resp., \eqref{hNLH}) possesses a ground state $Q_\epsilon^J\in H_{P_\epsilon}^1$.
\end{prop}

Throughout this section, we denote the order of nonlinearity by
\begin{align*}
p=\left\{\begin{aligned}
&2k+1&&\textup{for NLS }\eqref{hNLS}\\
&3&&\textup{for NLH }\eqref{hNLH}.
\end{aligned}\right.
\end{align*}
We observe that by algebra,
$$\langle\mathcal{N}'(u),u\rangle_{L^2}=(p+1)\mathcal{N}(u).$$
Hence, if $u$ is admissible for the variational problem \eqref{ground state energy level}, equivalently 
\begin{equation}\label{constraint}
0=\langle I_\epsilon'(u),u\rangle_{L^2}=\|u\|_{H_{P_\epsilon}^1}^2-(p+1)\mathcal{N}(u)\quad\left(\Leftrightarrow\|u\|_{H_{P_\epsilon}^1}^2=(p+1)\mathcal{N}(u)\right),
\end{equation}
then the action functional can be written as
\begin{equation}\label{action functional under constraint}
I_\epsilon(u)=\frac{1}{2}\|u\|_{H_{P_\epsilon}^1}^2-\mathcal{N}(u)=\frac{p-1}{2(p+1)}\|u\|_{H_{P_\epsilon}^1}^2=\frac{p-1}{2}\mathcal{N}(u).
\end{equation}

For each $\epsilon\geq0$ (including $\epsilon=0$), the \textit{ground state energy level} is defined by
\begin{equation}\label{ground state energy level}
\mathcal{C}_\epsilon := \inf\Big\{ I_\epsilon(u) ~|~ u \in H^1_{P_\epsilon} \setminus \{0\}\textup{ and }\langle I_\epsilon'(u),u \rangle_{L^2} = 0\Big\}.
\end{equation}
The following lemma is useful to prove the proposition.

\begin{lem}\label{energy level lemma}
Suppose that \eqref{uniform lower bound} (as well as \eqref{subcritical} for NLS) holds. Then, $\mathcal{C}_\epsilon$ is strictly positive, and 
\begin{equation}\label{energy level upper bound}
\limsup_{\epsilon\to0}\mathcal{C}_\epsilon\leq\mathcal{C}_0.
\end{equation}
\end{lem}

\begin{proof}
For NLS \eqref{hNLS}, by the Sobolev inequality with \eqref{uniform lower bound}, we have
\begin{equation}\label{nonlinear bound 1}
\mathcal{N}(u)=\frac{1}{2k+2}\|u\|_{L^{2k+2}}^{2k+2}\leq C' \|u\|_{H_{P_\epsilon}^1}^{2k+2},
\end{equation}
while for NLH \eqref{hNLH}, by the Hardy-Littlewood-Sobolev inequality and the Sobolev inequality with \eqref{uniform lower bound},
\begin{equation}\label{nonlinear bound 2}
\mathcal{N}(u)\leq \frac{1}{4}\left\|\frac{1}{|x|}*|u|^2\right\|_{L^6}\left\||u|^2\right\|_{L^{6/5}}\leq C \left\||u|^2\right\|_{L^{6/5}}^2=C \|u\|_{L^{12/5}}^4\leq C' \|u\|_{H_{P_\epsilon}^1}^4.
\end{equation}
Then, inserting the above inequality to the constraint \eqref{constraint}, we get
$$0\geq \|u\|_{H_{P_\epsilon}^1}^2(1-(p+1)C' \|u\|_{H_{P_\epsilon}^1}^{p-2})\quad\left(\Leftrightarrow\|u\|_{H_{P_\epsilon}^1}^2\geq ((p+1)C')^{-\frac{2}{p-2}} \right)$$
Thus, by \eqref{action functional under constraint}, $I(u_\epsilon)=\frac{p-1}{2(p+1)}\|u\|_{H_{P_\epsilon}^1}^2\geq\frac{p-1}{2(p+1)}((p+1)C')^{-\frac{2}{p-2}}$. Taking the infimum, we prove the lower bound on $\mathcal{C}_\epsilon$.

To show \eqref{energy level upper bound}, we observe that the ground state $Q_0$ for the limit equation \eqref{reference NLS} (resp., \eqref{reference NLH}) is almost admissible for the variational problem \eqref{ground state energy level} for sufficiently small $\epsilon>0$, because
\begin{align*}
\langle I_\epsilon'(Q_0),Q_0\rangle_{L^2}&=\langle (P_\epsilon+1)Q_0-\mathcal{N}'(Q_0), Q_0\rangle_{L^2}\\
&=\langle (-\Delta+1)Q_0-\mathcal{N}'(Q_0), Q_0\rangle_{L^2}+\langle (P_\epsilon-(-\Delta))Q_0, Q_0\rangle_{L^2}\\
&=0+o_\epsilon(1)=o_\epsilon(1),
\end{align*}
where in the third identity, we used that $Q_0\in H^\ell$ for all $\ell\in\mathbb{N}$. Hence, for each $\epsilon>0$, there exists $t_\epsilon=1+o_\epsilon(1)$ such that $t_\epsilon Q_0$ is admissible. Then, it follows from the definition of the level set $\mathcal{C}_\epsilon$ and \eqref{action functional under constraint} that
$$\mathcal{C}_\epsilon\leq I_\epsilon(t_\epsilon Q_0)=\frac{p-1}{2}\mathcal{N}(t_\epsilon Q_0)=t_\epsilon^{p+1}\cdot\frac{p-1}{2}\mathcal{N}(Q_0)=(1+o_\epsilon(1))\cdot I_0(Q_0).$$
Thus, taking $\limsup_{\epsilon\to 0}$, we prove \eqref{energy level upper bound}.
\end{proof}

\begin{proof}[Proof of Proposition \ref{existence}]
Let $\{u_n\}_{n=1}^\infty\subset H^1_{P_\epsilon}$ be a minimizing sequence for $I_\ep(u)$ subject to the constraint $\langle I_\epsilon'(u),u \rangle_{L^2} = 0$ with $u \neq 0$, which is, by \eqref{action functional under constraint}, bounded in $H_{P_\epsilon}^1$. We consider the Levy concentration function of $u_n$ (see \cite{Li})
\[
M_n(r) := \sup_{x\in \R^d}\int_{B_r(x)}|u_n|^2\,dx,
\]
where $B_r(x)$ denotes the Euclidean ball of the radius $r$ centered at $x$.

Suppose that there exists some $r > 0$ such that $M_n(r) \to 0$ as $n \to \infty$. It is shown in \cite{Li} that $\{u_n\}_{n=1}^\infty$ converges to zero in $L^p(\mathbb{R}^d)$ for every $2 < p < 2^*$, where $\frac{1}{2^*}=\max\{\frac{d-2}{2d},0\}$. Thus, by \eqref{action functional under constraint} and \eqref{nonlinear bound 1} (resp., \eqref{nonlinear bound 2}), it follows that $I_\epsilon(u_n)=\frac{p-1}{2}\mathcal{N}(u_n)\to 0$ as $n \to \infty$, but this contradicts to that $\mathcal{C}_\epsilon>0$ (see Lemma \ref{energy level lemma}).

Now, passing to a subsequence, we assume that
$$M_0 := \lim_{n\to\infty} M(1) > 0.$$ 
Then, there exists a  sequence $\{x_n\}_{n=1}^\infty \subset \R^d$ such that for sufficiently large $n$,
\begin{equation}\label{nontriviality}
\int_{B_1(x_n)}|u_n|^2\,dx > \frac{M_0}{2}.
\end{equation}
Translating the sequence, we introduce another minimizing sequence $\{v_n\}_{n=1}^\infty$ given by $v_n(x) = u_n(x+x_n)$, which is bounded in $H^1_{P_\epsilon}$. Let $v_0$ be the weak subsequential limit of $\{v_n\}_{n=1}^\infty$ in $H^1_{P_\epsilon}$ as $n\to\infty$. Note that $v_0\neq0$, since $\{v_n\}_{n=1}^\infty$ is locally compact in $L^2(\mathbb{R}^d)$ and it satisfies \eqref{nontriviality}.

We claim that $v_0$ is admissible for the minimization problem \eqref{ground state energy level}, i.e., 
\[
\langle I_\epsilon'(v_0), v_0 \rangle_{L^2} = 0.
\]
In order to prove the claim by contradiction, we assume that
$$\delta := \langle I_\epsilon'(v_0), v_0 \rangle_{L^2} > 0,$$
and then applying the well-known Brezis-Lieb lemma, we decompose $v_n = v_0 +w_n$ such that
\begin{align}
\|v_n\|_{H^1_{P_\epsilon}}^2 &= \|v_0\|_{H^1_{P_\epsilon}}^2+\|w_n\|_{H^1_{P_\epsilon}}^2 +o_n(1),\label{BL lemma}\\
\mathcal{N}(v_n) &= \mathcal{N}(v_0) +\mathcal{N}(w_n)+o_n(1),\nonumber
\end{align}
and consequently, 
$$\langle I_\epsilon'(w_n), w_n \rangle_{L^2} =\langle I_\epsilon'(v_n), v_n \rangle_{L^2} -\langle I_\epsilon'(v_0), v_0 \rangle_{L^2} +o_n(1)= -\delta +o_n(1).$$
We observe that $f_n(t) := \langle I_\epsilon'(tw_n), tw_n \rangle_{L^2}$ is a polynomial of the form $a_nt^2 -b_nt^{p+1}$ with $a_n, b_n > 0$, and that $f_n(1)\leq-\frac{\delta}{2}$ for large $n$. Hence, there exist a small $\eta \in (0, 1)$ and a sequence $\{t_n\}_{n=1}^\infty$, with $0< t_n \leq 1-\eta$, such that $\{t_nw_n\}_{n=1}^\infty$ is admissible, i.e., $\langle I_\epsilon'(t_nw_n), t_nw_n \rangle_{L^2} = 0$. Then, by \eqref{action functional under constraint} and \eqref{BL lemma}, we prove that 
\begin{align*}
I_\epsilon(t_nw_n)&= \frac{p-1}{2(p+1)}\|t_nw_n\|_{H^1_{P_\epsilon}}^2= t_n^2\cdot\frac{p-1}{2(p+1)}\|w_n\|_{H^1_{P_\epsilon}}^2\\
&\leq (1-\eta)^2 \cdot\frac{p-1}{2(p+1)}\|v_n\|_{H^1_{P_\epsilon}}^2+o_n(1)\\
&=(1-\eta)^2 I_\epsilon(v_n)+o_n(1)=(1-\eta)^2\mathcal{C}_\epsilon +o_n(1).
\end{align*}
However, this contradicts to minimality of $\mathcal{C}_\epsilon$. If $\delta := \langle I_\epsilon'(v_0), v_0 \rangle_{L^2} < 0$, repeating the same argument but switching the role of $v_0$ with $w_n$, we can again deduce a contradiction. Therefore, the claim is proved.

Finally, by the lower semi-continuity of the norm $\|\cdot\|_{H^1_{P_\epsilon}}$, we show that $v_0$ achieves the minimal energy, 
$$I_\epsilon(v_0) = \frac{p-1}{2(p+1)}\|v_0\|_{H^1_{P_\epsilon}} \leq \frac{p-1}{2(p+1)}\lim_{n\to\infty}\|v_n\|_{H^1_{P_\epsilon}} = \lim_{n\to\infty}I_\epsilon(v_n) = \mathcal{C}_\epsilon.$$
This completes the proof by setting $Q_\epsilon^J := v_0$. 
\end{proof}

\begin{prop}[Convergence of ground states]\label{convergence}
Suppose that \eqref{uniform lower bound} (as well as \eqref{subcritical} for NLS) holds. Let $\{Q_\epsilon^J\}_{\epsilon>0}$ be the family of ground states for the higher-order equation \eqref{hNLS} (resp., \eqref{hNLH}) given by Proposition \ref{existence}.
Then,
$$\lim_{\epsilon\to 0} \|Q_\epsilon^J-\tilde{Q}_0\|_{H_{P_\epsilon}^1}=0,$$
where $\tilde{Q}_0$ is a ground state to the second-order equation \eqref{reference NLS} (resp., \eqref{reference NLH}). 
\end{prop}

\begin{proof}
By \eqref{uniform lower bound}, \eqref{action functional under constraint} and Lemma \ref{ground state energy level}, we see that $\{Q_\epsilon^J\}_{\epsilon>0}$ is bounded in $H^1$, 
$$\gamma\|Q_\epsilon^J\|_{H^1}^2 \leq \|Q_\epsilon^J\|_{H^1_{P_\epsilon}}^2 = \frac{2(p+1)}{p-1}I_\epsilon(Q_\epsilon^J) = \frac{2(p+1)}{p-1}\mathcal{C}_\epsilon=\frac{2(p+1)}{p-1}\mathcal{C}_0+o_\epsilon(1).$$
Hence, $Q_\epsilon^J$ weakly subsequentially converges to $\tilde{Q}_0$ in $H^1$ as $\epsilon\to 0$. As in the proof of Proposition \ref{existence}, one can show that 
\[
\int_{B_1(0)}|Q_\epsilon^J|^2\,dx \geq \frac{M_0}{2},
\]
which implies $\tilde{Q}_0$ is nontrivial. 

We claim that $\tilde{Q}_0$ is a smooth solution to \eqref{reference NLS} (resp., \eqref{reference NLH}). To show the claim, we recall that for any $\phi \in C_c^\infty$,
$$\langle (P_\epsilon+1) Q_\epsilon^J- \mathcal{N}'(Q_\epsilon^J), \phi \rangle_{L^2} = 0.$$
However, by the weak convergence of $Q_\epsilon^J$, up to a subsequence, we have
$$\lim_{\ep \to 0}\langle (P_\epsilon+1) Q_\epsilon^J, \phi \rangle_{L^2} = \lim_{\ep \to 0}\langle  Q_\epsilon^J, (P_\epsilon+1)\phi \rangle_{L^2}= \lim_{\ep \to 0}\langle  Q_\epsilon^J, (-\Delta+1)\phi \rangle_{L^2}= \langle  \tilde{Q}_0, (-\Delta+1)\phi \rangle_{L^2} $$
and
$$\lim_{\epsilon\to0}\langle \mathcal{N}'(Q_\epsilon), \phi \rangle_{L^2} = \langle \mathcal{N}'(\tilde{Q}_0), \phi \rangle_{L^2}.$$
Thus, sending $\epsilon\to 0$, we show that 
$$\langle (-\Delta+1) \tilde{Q}_0- \mathcal{N}'(\tilde{Q}_0), \phi \rangle_{L^2} = 0.$$
In other words, $\tilde{Q}_0$ is a weak solution to \eqref{reference NLS} (resp., \eqref{reference NLH}). Then, by the elliptic regularity (see \cite{GT}, one can show that $\tilde{Q}_0 \in H^\ell$ for every $\ell\in \mathbb{N}$.

Next, using \eqref{action functional under constraint}, we write 
\begin{equation}\label{convergence proof}
\begin{aligned}
\mathcal{C}_0 &\leq I_0(\tilde{Q}_0) = \frac{p-1}{2(p+1)}\|\tilde{Q}_0\|_{H^1}^2\\
&\leq \frac{p-1}{2(p+1)}\left(\|\tilde{Q}_0\|_{H^1_{P_\epsilon}}^2 +\|Q_\epsilon^J-\tilde{Q}_0\|_{H^1_{P_\epsilon}}^2\right)\\
&= \frac{p-1}{2(p+1)}\|Q_\epsilon^J\|_{H^1_{P_\epsilon}}^2-\frac{p-1}{p+1}\cdot\textup{Re}\langle \tilde{Q}_0, Q_\epsilon^J-\tilde{Q}_0\rangle_{H^1_{P_\epsilon}}\\
&= I_\epsilon(Q_\epsilon^J)-\frac{p-1}{p+1}\cdot\textup{Re}\langle \tilde{Q}_0, Q_\epsilon^J-\tilde{Q}_0\rangle_{H^1_{P_\epsilon}}.
\end{aligned}
\end{equation}
However, since $\tilde{Q}_0$ is smooth and $Q_\epsilon^J\rightharpoonup\tilde{Q}_0$ in $H^1$ as $\epsilon\to 0$, we have
$$\langle \tilde{Q}_0, Q_\epsilon^J-\tilde{Q}_0\rangle_{H^1_{P_\epsilon}}=\langle \tilde{Q}_0, Q_\epsilon^J-\tilde{Q}_0\rangle_{H^1}+\langle (P_\epsilon-(-\Delta))\tilde{Q}_0, Q_\epsilon^J-\tilde{Q}_0\rangle_{L^2}\to 0$$
as $\epsilon\to 0$. Thus, by \eqref{action functional under constraint} again and Lemma \ref{energy level lemma}, we get
$$\mathcal{C}_0\leq \mathcal{C}_\epsilon+o_\epsilon(1)\leq \mathcal{C}_0+o_\epsilon(1).$$
Sending $\epsilon\to 0$ in \eqref{convergence proof}, we conclude that $\tilde{Q}_0$ achieves the minimum value $\mathcal{C}_0$ of the action functional $I_0$ and that $\|Q_\epsilon^J-\tilde{Q}_0\|_{H^1_{P_\epsilon}}\to 0$ as $\epsilon\to 0$.
\end{proof}

\begin{rem}
Let $Q_0$ be the radially symmetric positive ground state for  \eqref{reference NLS} (resp., \eqref{reference NLH}). By uniqueness of a ground state to the second-order equation, there exist $\theta \in \mathbb{R}$ and $x_0 \in \mathbb{R}^d$ such that
$Q_0(x) = e^{i\theta}\tilde{Q}_0(x-x_0)$. Then, the modified profile $e^{i\theta}Q_\epsilon^J(\cdot-x_0)$, which is also a ground state, converges to $Q_0$ as $\epsilon\to0$.
\end{rem}

\begin{prop}[Upgraded convergence]\label{upgraded convergence}
Suppose that \eqref{uniform lower bound} (as well as \eqref{subcritical} for NLS) holds. For $\epsilon>0$, let $u_\epsilon\in H_{P_\epsilon}^1$ be a solution to the higher-order equation \eqref{hNLS} (resp., \eqref{hNLH}), which is not necessarily a ground state. Let $Q_0$ be the unique ground state to the second-order equation \eqref{reference NLS} (resp., \eqref{reference NLH}).
\begin{enumerate}
\item If $u_\epsilon\to Q_0$ in $H^1$ as $\epsilon\to 0$, then $u_\epsilon\to Q_0$ in $H^\ell$ as $\epsilon\to 0$ for all $\ell\in\mathbb{N}$.
\item As a consequence, for all $\ell\in\mathbb{N}$, $\|u_\epsilon\|_{H^\ell}$ is bounded uniformly in $0<\epsilon\leq 1$.
\end{enumerate}
\end{prop}

\begin{proof}
We prove the proposition by induction. Let $r_\epsilon=u_\epsilon-Q_0$ be the difference between two solutions. Suppose that $\|r_\epsilon\|_{H^\ell}\to 0$ for some $\ell\in\mathbb{N}$. Then, by the equations, we have
\begin{align*}
(P_\epsilon+1)r_\epsilon&=(-\Delta-P_\epsilon)Q_0+(P_\epsilon+1)u_\epsilon-(-\Delta+1)Q_0\\
&=(-\Delta-P_\epsilon)Q_0+\mathcal{N}'(u_\epsilon)-\mathcal{N}'(Q_0).
\end{align*}
Thus, it follows from \eqref{uniform lower bound} that 
\begin{align*}
\gamma\|r_\epsilon\|_{H^{\ell+1}}&\leq\|(P_\epsilon+1)r_\epsilon\|_{H^{\ell-1}}\\
&\leq\|(-\Delta-P_\epsilon)Q_0\|_{H^{\ell-1}}+\|\mathcal{N}'(Q_0+r_\epsilon)-\mathcal{N}'(Q_0)\|_{H^{\ell-1}}.
\end{align*}
For the first term on the right hand side, by smoothness of $Q_0$, $\|(-\Delta-P_\epsilon)Q_0\|_{H^{\ell-1}}\to 0$. For the second term, distributing derivatives and then applying Lemma \ref{multilinear estimate}, one can show that
$$\|\mathcal{N}'(Q_0+r_\epsilon)-\mathcal{N}'(Q_0)\|_{H^{\ell-1}}\lesssim\Big\{\|Q_0\|_{H^\ell}+\|r_\epsilon\|_{H^\ell}\Big\}^{p-1}\|r_\epsilon\|_{H^\ell}\to 0,$$
where $p=2k+1$ for \eqref{hNLS} (resp., $p=3$ for \eqref{hNLH}). Therefore, we conclude that $\|r_\epsilon\|_{H^{\ell+1}}\to0$.
\end{proof}

\section{Non-degeneracy estimates}

Let $\{u_\epsilon\}_{\epsilon>0}$ be a family of real-valued solutions to the higher-order equation such that $u_\epsilon\to Q_0$ in $H^1$ (as well as in $H^\ell$ for all $\ell\in\mathbb{N}$ by Proposition \ref{upgraded convergence}) as $\epsilon\to 0$, where $Q_0$ is the unique radially symmetric positive ground state for the second-order equation \eqref{reference NLS} (resp., \eqref{reference NLH}). For notational convenience, we denote $u_0:=Q_0.$ For $0\leq\epsilon\leq\epsilon_0$ (including 0), we consider the linear operators $\mathcal{L}_\epsilon^\pm: H_{P_\epsilon}^1\to H_{P_\epsilon}^{-1}$, defined by 
\begin{equation}\label{L operator}
\begin{aligned}
\left\{\begin{aligned}
\mathcal{L}_\epsilon^+:&=P_\epsilon+1-\mathcal{N}_{u_\epsilon}^+\\
\mathcal{L}_\epsilon^-:&=P_\epsilon+1-\mathcal{N}_{u_\epsilon}^-,
\end{aligned}\right.
\end{aligned}
\end{equation}
where
\begin{equation}\label{N+}
\mathcal{N}_{u}^+(g):=\left\{\begin{aligned}
&(2k+1)u^{2k}g&&\textup{(for NLS)}\\
&2(|x|^{-1}*(u g))u+(|x|^{-1}*u^2)g&&\textup{(for NLH)}
\end{aligned}\right.
\end{equation}
and
\begin{equation}\label{N-}
\mathcal{N}_{u}^-(g):=\left\{\begin{aligned}
&u^{2k}g&&\textup{(for NLS)}\\
&(|x|^{-1}*u^2)g&&\textup{(for NLH)}.
\end{aligned}\right.
\end{equation}
These linear operators naturally appear as the real and the imaginary parts of the linearized operator at the solution $u_\epsilon$. Factorizing out the differential operator $(1+P_\epsilon)$ in a symmetric form, we write 
\begin{equation}\label{LA relation}
\mathcal{L}_\epsilon^\pm=\sqrt{1+P_\epsilon}(\textup{Id}-\mathcal{A}_\epsilon^\pm) \sqrt{1+P_\epsilon},
\end{equation}
where
\begin{align*}
\left\{\begin{aligned}
\mathcal{A}_\epsilon^+:&=\frac{1}{\sqrt{1+P_\epsilon}}\mathcal{N}_{u_\epsilon}^+\frac{1}{\sqrt{1+P_\epsilon}}\\
\mathcal{A}_\epsilon^-:&=\frac{1}{\sqrt{1+P_\epsilon}}\mathcal{N}_{u_\epsilon}^-\frac{1}{\sqrt{1+P_\epsilon}}.
\end{aligned}\right.
\end{align*}
In this section, we prove non-degeneracy of the solution $u_\epsilon$, and obtain uniform lower bounds for the linear operators $(\textup{Id}-\mathcal{A}_\epsilon^\pm)$, which are our main analytic tools.

To begin with, we consider the base case $\epsilon=0$. By the non-degeneracy of the ground state $Q_0$ and the relation \eqref{LA relation}, we have
$$\textup{Ker}(\textup{Id}-\mathcal{A}_0^+)=\textup{span}\Big\{\partial_{x_1}\sqrt{1-\Delta}\,Q_0,\cdots,\partial_{x_d}\sqrt{1-\Delta}\,Q_0\Big\}$$
and
$$\textup{Ker}(\textup{Id}-\mathcal{A}_0^-)=\textup{span}\Big\{\sqrt{1-\Delta}\,Q_0\Big\}.$$
By the equation, the operator $\mathcal{A}_0^\pm$ sends an element of $\textup{Ker}(\textup{Id}-\mathcal{A}_0^\pm)$ to the same function, and thus $(\textup{Id}-\mathcal{A}_0^\pm)$ maps $(\textup{Ker}(\textup{Id}-\mathcal{A}_0^\pm))^\perp\subset L^2(\mathbb{R}^d;\mathbb{R})$ to itself, where $A^\perp\subset H$ denotes the subspace orthogonal to $A$ in the Hilbert space $H$. Moreover, the operator $(\textup{Id}-\mathcal{A}_0^\pm) $ satisfies the following lower bounds.

\begin{prop}[Non-degeneracy estimates; base case]\label{non-degeneracy estimate 0}
There exists $\beta_0>0$ such that
$$\|(\textup{Id}-\mathcal{A}_0^\pm) g\|_{L^2(\mathbb{R}^d;\mathbb{R})}\geq \beta_0\|g\|_{L^2(\mathbb{R}^d;\mathbb{R})}$$
for all $g\in(\textup{Ker}(\textup{Id}-\mathcal{A}_0^\pm))^\perp\subset L^2(\mathbb{R}^d;\mathbb{R})$.
\end{prop}

\begin{proof}
We claim that both $\frac{1}{\sqrt{1-\Delta}}\mathcal{N}_{Q_0}^+\frac{1}{\sqrt{1-\Delta}}$ and $\frac{1}{\sqrt{1-\Delta}}\mathcal{N}_{Q_0}^-\frac{1}{\sqrt{1-\Delta}}$ are compact on $L^2$. Indeed, for NLS, the integral kernel of $\mathcal{N}_{Q_0}^+\frac{1}{\sqrt{1-\Delta}}$ (or $\mathcal{N}_{Q_0}^-\frac{1}{\sqrt{1-\Delta}}$, respectively) is given by
$$(2k+1)Q_0(x)^{2k}G_{-1}(x-y)\quad\Big(\textup{or }Q_0(x)^{2k}G_{-1}(x-y)\textup{, respectively}\Big),$$
where $G_{-1}(x)=((1+|\xi|^2)^{-1})^\vee(x)$ is the Bessel potential. For NLH, it is given by
\begin{align*}
&2\int_{\mathbb{R}^3}\frac{Q_0(x)Q_0(z)}{|x-z|}G_{-1}(z-y)dz+(|x|^{-1}*Q_0^2)(x)G_{-1}(x-y)\\
&\Big(\textup{or }(|x|^{-1}*Q_0^2)(x)G_{-1}(x-y)\textup{, respectively}\Big).
\end{align*}
All of the above kernels are contained in $L^2(\mathbb{R}_x^d\times\mathbb{R}_y^d)$, because $Q_0$ is smooth and rapidly decaying. Therefore, the associated operators are Hilbert-Schmidt (so,  compact on $L^2(\mathbb{R}^d)$). Since composition of a compact operator and a bounded operator is compact, this proves the claim. As a consequence, by the Fredholm alternative, the proposition is proved.
\end{proof}

Next, we show that the non-degeneracy of the ground state $Q_0$ is stable along the family of solutions which converges to the ground state $Q_0$.
\begin{prop}[Stability of non-degeneracy]\label{Non-degeneracy epsilon}
Let $\{u_\epsilon\}_{\epsilon>0}$ be a family of real-valued solutions to the higher-order equation \eqref{hNLS} (resp., \eqref{hNLH}) such that $u_\epsilon\to Q_0$ in $H^1$ as $\epsilon\to 0$.  Then, there exists $\epsilon_0>0$ such that
$$\textup{Ker}\mathcal{L}_\epsilon^+=\textup{span}\Big\{\partial_{x_1}u_\epsilon,\cdots,\partial_{x_d}u_\epsilon\Big\}\quad\textup{and}\quad
\textup{Ker}\mathcal{L}_\epsilon^-=\textup{span}\big\{u_\epsilon\big\}$$
for $0<\epsilon\leq\epsilon_0$. Equivalently, we have
\begin{equation}\label{eq: Non-degeneracy epsilon}
\textup{Ker}(\textup{Id}-\mathcal{A}_\epsilon^+)=\textup{span}\Big\{\partial_{x_1}\sqrt{1+P_\epsilon}\,u_\epsilon,\cdots,\partial_{x_d}\sqrt{1+P_\epsilon}\,u_\epsilon\Big\}
\end{equation}
and
\begin{equation}\label{eq: Non-degeneracy epsilon'}
\textup{Ker}(\textup{Id}-\mathcal{A}_\epsilon^-)=\textup{span}\Big\{\sqrt{1+P_\epsilon}\,u_\epsilon\Big\}.
\end{equation}
\end{prop}

\begin{proof}
Following the argument in the proof of \cite[Theorem 3]{L2}, we prove \eqref{eq: Non-degeneracy epsilon} only, because \eqref{eq: Non-degeneracy epsilon'} can be proved by the same way.

By the equation, it is easy to see that each $\partial_{x_j}\sqrt{1+P_\epsilon}u_\epsilon$ is contained in the kernel of $(\textup{Id}-\mathcal{A}_\epsilon^+)$. Therefore, it suffices to show that the dimension of $\textup{Ker}(\textup{Id}-\mathcal{A}_\epsilon^+)$ is $\leq d$. We recall that $\textup{Ker}(\textup{Id}-\mathcal{A}_\epsilon^+)=\textup{Im}(\mathcal{P}_\epsilon)$, where $\mathcal{P}_\epsilon$ is the projection operator given by
$$\mathcal{P}_\epsilon:=\frac{1}{2\pi i}\oint_{|z|=c} (\textup{Id}-\mathcal{A}_\epsilon^+-z\textup{Id})^{-1}dz$$
for some sufficiently small $c>0$. We observe that by Lemma \ref{multilinear estimate},
\begin{equation}\label{A convergence}
\begin{aligned}
\|\mathcal{A}_\epsilon^+-\mathcal{A}_0^+\|_{L^2\to L^2} &=\left\|\frac{1}{\sqrt{1+P_\epsilon}}\mathcal{N}_{u_\epsilon}^+\frac{1}{\sqrt{1+P_\epsilon}}-\frac{1}{\sqrt{1-\Delta}}\mathcal{N}_{Q_0}^+\frac{1}{\sqrt{1-\Delta}}\right\|_{L^2\to L^2}\\
&\leq\left\|\left(\frac{1}{\sqrt{1+P_\epsilon}}-\frac{1}{\sqrt{1-\Delta}}\right)\mathcal{N}_{u_\epsilon}^+\frac{1}{\sqrt{1+P_\epsilon}}\right\|_{L^2\to L^2}\\
&\quad+\left\|\frac{1}{\sqrt{1-\Delta}}(\mathcal{N}_{u_\epsilon}^+-\mathcal{N}_{Q_0}^+)\frac{1}{\sqrt{1+P_\epsilon}}\right\|_{L^2\to L^2}\\
&\quad+\left\|\frac{1}{\sqrt{1-\Delta}}\mathcal{N}_{Q_0}^+\left(\frac{1}{\sqrt{1+P_\epsilon}}-\frac{1}{\sqrt{1-\Delta}}\right)\right\|_{L^2\to L^2}\to 0,
\end{aligned}
\end{equation}
and consequently, $\|\mathcal{P}_\epsilon-\mathcal{P}_0\|_{L^2\to L^2}\to 0$ as $\epsilon\to0$. Suppose that $\textup{Rank}(\mathcal{P}_\epsilon)>\textup{Rank}(\mathcal{P}_0)$. Then, there exist $L^2$-orthonormal vectors $v_1,\cdots, v_{d+1}$ such that $\mathcal{P}_\epsilon v_j=v_j$. Hence, $\mathcal{P}_0v_1,\cdots, \mathcal{P}_0v_{d+1}$ are almost orthogonal, and they are linearly independent, which contradicts to the assumption. Therefore, we conclude that $\textup{Rank}(\mathcal{P}_\epsilon)\leq\textup{Rank}(\mathcal{P}_0)=d$.
\end{proof}

Using the non-degeneracy, we prove the inequality analogous to Proposition \ref{non-degeneracy estimate 0}.

\begin{prop}[Non-degeneracy estimates; general case]\label{prop: non-degeneracy estimates}
Let $\{u_\epsilon\}_{\epsilon>0}$ be a family of real-valued solutions to the higher-order equation \eqref{hNLS} (resp., \eqref{hNLH}) such that $u_\epsilon\to Q_0$ in $H^1$ as $\epsilon\to 0$. Then, there exist $\epsilon_0>0$ and $\beta>0$ such that if $0<\epsilon\leq\epsilon_0$, then
$$\|(\textup{Id}-\mathcal{A}_\epsilon^\pm) g\|_{L^2(\mathbb{R}^d;\mathbb{R})}\geq \beta\|g\|_{L^2(\mathbb{R}^d;\mathbb{R})}$$
for all $g\in (\textup{Ker}(\textup{Id}-\mathcal{A}_\epsilon^\pm))^\perp\subset L^2(\mathbb{R}^d;\mathbb{R})$, which is equivalent to
$$\|\mathcal{L}_\epsilon^\pm g\|_{H_{P_\epsilon}^{-1}(\mathbb{R}^d;\mathbb{R})}\geq \beta\|g\|_{H_{P_\epsilon}^1(\mathbb{R}^d;\mathbb{R})}$$
for all $g\in(\textup{Ker}\mathcal{L}_\epsilon^\pm)^\perp\subset H_{P_\epsilon}^1(\mathbb{R}^d;\mathbb{R})$.
\end{prop}

\begin{proof}
We show the proposition only for $\mathcal{A}_\epsilon^+$, since the other inequality can be proved exactly by the same way.

Let $\beta=\frac{\beta_0}{4}>0$, where $\beta_0$ is given in Proposition \ref{non-degeneracy estimate 0}. For $g\in L^2(\mathbb{R}^d;\mathbb{R})$ and $\epsilon\geq0$, we denote by $g_\epsilon^\perp$ the orthogonal projection of $g$ to $(\textup{Ker}(\textup{Id}-\mathcal{A}_\epsilon^+))^\perp\subset L^2(\mathbb{R}^d;\mathbb{R})$, precisely
$$g_\epsilon^\perp:=g-\sum_{j=1}^d\langle g, e_{j;\epsilon}\rangle_{L^2}e_{j;\epsilon},$$
where $e_{j;\epsilon}:=\frac{\partial_{x_j}\sqrt{1+P_\epsilon}u_\epsilon}{\|\partial_{x_j}\sqrt{1+P_\epsilon}u_\epsilon\|_{L^2}^2}$. We fix $g\in (\textup{Ker}(\textup{Id}-\mathcal{A}_\epsilon^+))^\perp$. Then, we decompose
\begin{align*}
(\textup{Id}-\mathcal{A}_\epsilon^+)g&=(\textup{Id}-\mathcal{A}_0^+)g+(\mathcal{A}_0^+-\mathcal{A}_\epsilon^+)g\\
&=(\textup{Id}-\mathcal{A}_0^+)g_0^\perp+(\textup{Id}-\mathcal{A}_0^+)(g-g_0^\perp)+(\mathcal{A}_0^+-\mathcal{A}_\epsilon^+)g.
\end{align*}
By the triangle inequalities and Proposition \ref{non-degeneracy estimate 0}, we get
\begin{equation}\label{proof: non-degeneracy estimates}
\begin{aligned}
\|(\textup{Id}-\mathcal{A}_\epsilon^+)g\|_{L^2}&\geq\|(\textup{Id}-\mathcal{A}_0^+)g_0^\perp\|_{L^2}-\|(\textup{Id}-\mathcal{A}_0^+)(g-g_0^\perp)\|_{L^2}-\|(\mathcal{A}_0^+-\mathcal{A}_\epsilon^+)g\|_{L^2}\\
&\geq 4\beta\|g_0^\perp\|_{L^2}-\|g-g_0^\perp\|_{L^2}-\|\mathcal{A}_0^+(g-g_0^\perp)\|_{L^2}-\|(\mathcal{A}_0^+-\mathcal{A}_\epsilon^+)g\|_{L^2}\\
&\geq 4\beta\|g\|_{L^2}-(4\beta+1+\|\mathcal{A}_0^+\|_{L^2\to L^2})\|g_\epsilon^\perp-g_0^\perp\|_{L^2}\quad\textup{(by $g_\epsilon^\perp=g$)}\\
&\quad-\|\mathcal{A}_0^+-\mathcal{A}_\epsilon^+\|_{L^2\to L^2}\|g\|_{L^2}.
\end{aligned}
\end{equation}
On the other hand, we have
\begin{align*}
\|g_\epsilon^\perp-g_0^\perp\|_{L^2}&\leq\sum_{j=1}^d\left\|\langle g,e_{j;\epsilon}\rangle_{L^2}e_{j;\epsilon}-\langle g,e_{j;0}\rangle_{L^2}e_{j;0}\right\|_{L^2}\\
&\leq\sum_{j=1}^d|\langle g,e_{j;\epsilon}-e_{j;0}\rangle_{L^2}|+|\langle g,e_{j;0}\rangle_{L^2}|\left\|e_{j;\epsilon}-e_{j;0}\right\|_{L^2}\\
&\leq2\|g\|_{L^2}\sum_{j=1}^d\left\|e_{j;\epsilon}-e_{j;0}\right\|_{L^2}\leq o_\epsilon(1)\|g\|_{L^2},
\end{align*}
because by Proposition \ref{upgraded convergence}, 
\begin{align*}
&\left\|\partial_{x_j}\sqrt{1+P_\epsilon}u_\epsilon-\partial_{x_j}\sqrt{1-\Delta}u_0\right\|_{L^2}\\
&\leq\left\|(\sqrt{1+P_\epsilon}-\sqrt{1-\Delta})\partial_{x_j}u_\epsilon\right\|_{L^2}+\left\|\partial_{x_j}\sqrt{1-\Delta}(u_\epsilon-u_0)\right\|_{L^2}\to 0
\end{align*}
as $\epsilon\to 0$ and it implies $\|e_{j;\epsilon}-e_{j;0}\|_{L^2}\to 0$. Moreover, by \eqref{A convergence}, $\|\mathcal{A}_0^+-\mathcal{A}_\epsilon^+\|_{L^2\to L^2}\to 0$ as $\epsilon\to0$. Inserting these to \eqref{proof: non-degeneracy estimates}, we prove the proposition.
\end{proof}

By a little modification, we can also show the following inequality.

\begin{lem}\label{radial lower bound}
There exists $\epsilon_0>0$ such that if $0<\epsilon\leq\epsilon_0$, then
$$\textup{Id}-\frac{1}{\sqrt{1+P_\epsilon}}\mathcal{N}_{Q_0}^+\frac{1}{\sqrt{1+P_\epsilon}}$$
is invertible on $L_{rad}^2(\mathbb{R}^d;\mathbb{R})$. Moreover, its inverse is uniformly bounded,
$$\left\|\left(\textup{Id}-\frac{1}{\sqrt{1+P_\epsilon}}\mathcal{N}_{Q_0}^+\frac{1}{\sqrt{1+P_\epsilon}}\right)^{-1}\right\|_{L_{rad}^2(\mathbb{R}^d;\mathbb{R})\to L_{rad}^2(\mathbb{R}^d;\mathbb{R})}\leq\frac{2}{\beta_0},$$
where $\beta_0>0$ is given by Proposition \ref{non-degeneracy estimate 0}.
\end{lem}

\begin{proof}
By Proposition \ref{non-degeneracy estimate 0}, the operator $(\textup{Id}-\frac{1}{\sqrt{1-\Delta}}\mathcal{N}_{Q_0}^+\frac{1}{\sqrt{1-\Delta}})$ is invertible, because its kernel in $L_{rad}^2$ is empty ($\partial_{x_j}Q_0$'s are not radially symmetric). On the other hand, repeating the proof of  \eqref{A convergence}, one can show that the difference
\begin{align*}
&\frac{1}{\sqrt{1+P_\epsilon}}\mathcal{N}_{Q_0}^+\frac{1}{\sqrt{1+P_\epsilon}}-\frac{1}{\sqrt{1-\Delta}}\mathcal{N}_{Q_0}^+\frac{1}{\sqrt{1-\Delta}}\\
&=\left(\frac{1}{\sqrt{1+P_\epsilon}}-\frac{1}{\sqrt{1-\Delta}}\right)\mathcal{N}_{Q_0}^+\frac{1}{\sqrt{1+P_\epsilon}}+\frac{1}{\sqrt{1-\Delta}}\mathcal{N}_{Q_0}^+\left(\frac{1}{\sqrt{1+P_\epsilon}}-\frac{1}{\sqrt{1-\Delta}}\right)
\end{align*}
can be arbitrarily small in the operator norm on $L_{rad}^2(\mathbb{R}^d;\mathbb{R})$, provided that $\epsilon>0$ is small enough. Therefore, we conclude that if $0<\epsilon\leq\epsilon_0$, then $(\textup{Id}-\frac{1}{\sqrt{1+P_\epsilon}}\mathcal{N}_{Q_0}^+\frac{1}{\sqrt{1+P_\epsilon}})$ is invertible, and its inverse is uniformly bounded. 
\end{proof}

\section{Construction of a solution by contraction}\label{sec: construction}

In this section, by the contraction mapping argument in \cite{CHS2} which relies on the non-degeneracy estimates in the previous section, we construct a radially symmetric real-valued solution $u_\epsilon$ to the higher-order equation, with small $\epsilon>0$, that converges to the ground state $Q_0$ as $\epsilon\to0$, where $Q_0$ is the unique radially symmetric real-valued ground state for the second-order equation.

\begin{prop}[Construction of a solution by contraction]\label{contraction}
Suppose that \eqref{uniform lower bound} (as well as \eqref{subcritical} for NLS) holds. Then, there exists $\epsilon_0>0$ such that a sequence of radially symmetric real-valued solutions $\{u_\epsilon\}_{0<\epsilon\leq\epsilon_0}$ to the higher-order NLS \eqref{hNLS} (resp., the higher-order NLH \eqref{hNLH}) exists, with the convergence
$$\lim_{\epsilon\to0}\|u_\epsilon-Q_0\|_{H_{P_\epsilon}^1}=0.$$
\end{prop}

\begin{proof}
\underline{\textit{Step 1. Reformulation of the equation.}} Let $\epsilon>0$ be sufficiently small. Suppose that $u_\epsilon$ is a radially symmetric real-valued solution to the higher-order equation. Then, the difference
$$r_\epsilon:=u_\epsilon-Q_0$$
solves the equation
\begin{align*}
(P_\epsilon+1) r_\epsilon&=(P_\epsilon+1) u_\epsilon-(P_\epsilon+1) Q_0\\
&=(-\Delta-P_\epsilon)Q_0+(P_\epsilon+1) u_\epsilon-(-\Delta+1)Q_0\\
&=(-\Delta-P_\epsilon)Q_0+\mathcal{N}'(u_\epsilon)-\mathcal{N}'(Q_0)\\
&=(-\Delta-P_\epsilon)Q_0+\mathcal{N}'(Q_0+r_\epsilon)-\mathcal{N}'(Q_0).
\end{align*}
Moving the linear terms with respect to $r_\epsilon$ on the right hand side to the left, we write
$$(P_\epsilon+1-\mathcal{N}_{Q_0}^+) r_\epsilon=(-\Delta-P_\epsilon)Q_0+\mathcal{N}'(Q_0+r_\epsilon)-\mathcal{N}'(Q_0)-\mathcal{N}_{Q_0}^+(r_\epsilon)$$
(see \eqref{N+} for the definition of $\mathcal{N}_{Q_0}^+$). Then, inverting the operator
$$(P_\epsilon+1-\mathcal{N}_{Q_0}^+)=\sqrt{1+P_\epsilon}\left(\textup{Id}-\frac{1}{\sqrt{1+P_\epsilon}}\mathcal{N}_{Q_0}^+\frac{1}{\sqrt{1+P_\epsilon}}\right)\sqrt{1+P_\epsilon}$$
by Lemma \ref{radial lower bound}, we reformulate the higher-order equation as
\begin{align*}
r_\epsilon&=(P_\epsilon+1-\mathcal{N}_{Q_0}^+)^{-1}\Big\{(-\Delta-P_\epsilon)Q_0+\mathcal{N}'(Q_0+r_\epsilon)-\mathcal{N}'(Q_0)-\mathcal{N}_{Q_0}^+(r_\epsilon)\Big\}\\
&=:\Phi(r_\epsilon).
\end{align*}
\underline{\textit{Step 2. Construction of a solution.}} We set
$$\delta_\epsilon:=\frac{4}{\beta_0}\left\|(-\Delta-P_\epsilon)u_0\right\|_{H_{P_\epsilon}^{-1}},$$
where $\beta_0>0$ is the constant given in Lemma \ref{radial lower bound}. Then, by \eqref{uniform lower bound}, we have
$$\delta_\epsilon\leq\frac{4}{\beta_0}\sum_{|\alpha|=3}^J|c_\alpha| \epsilon^{|\alpha|-2}\|\nabla^\alpha Q_0\|_{H_{P_\epsilon}^{-1}}\leq \frac{4}{\beta_0\gamma}\sum_{|\alpha|=3}^J|c_\alpha| \epsilon^{|\alpha|-2}\|\nabla^\alpha Q_0\|_{H^{-2}}\to 0$$
as $\epsilon\to 0$. Let $\epsilon_0$ be a sufficiently small number to be chosen so that all of the following estimates hold. Suppose that $0<\epsilon\leq\epsilon_0$ (and thus $\delta_\epsilon>0$ is also small enough). If $\|r\|_{H_{P_\epsilon}^1}, \|\tilde{r}\|_{H_{P_\epsilon}^1}\leq\delta_\epsilon$, then by Lemma \ref{radial lower bound} and Lemma \ref{multilinear estimate}, 
\begin{align*}
\|\Phi(r)\|_{H_{P_\epsilon}^1}&=\Bigg\|\left(\textup{Id}-\frac{1}{\sqrt{1+P_\epsilon}}\mathcal{N}_{Q_0}^+\frac{1}{\sqrt{1+P_\epsilon}}\right)^{-1}\frac{1}{\sqrt{1+P_\epsilon}}\\
&\quad\quad\Big\{(-\Delta-P_\epsilon)Q_0+\mathcal{N}'(Q_0+r)-\mathcal{N}'(Q_0)-\mathcal{N}_{Q_0}^+(r)\Big\}\Bigg\|_{L^2}\\
&\leq \frac{2}{\beta_0}\left\|(-\Delta-P_\epsilon)Q_0\right\|_{H_{P_\epsilon}^{-1}}+\frac{2}{\beta_0}\left\|\mathcal{N}'(Q_0+r)-\mathcal{N}'(Q_0)-\mathcal{N}_{Q_0}^+(r)\right\|_{H_{P_\epsilon}^{-1}}\\
&\leq \frac{2}{\beta_0}\left\|(-\Delta-P_\epsilon)u_0\right\|_{H_{P_\epsilon}^{-1}}+\frac{1}{2}\|r\|_{H_{P_\epsilon}^1}\leq\delta_\epsilon
\end{align*}
and similarly, 
\begin{align*}
\|\Phi(r)-\Phi(\tilde{r})\|_{H_{P_\epsilon}^1}&\leq \frac{2}{\beta_0}\Big\|\left(\mathcal{N}'(Q_0+r)-\mathcal{N}'(Q_0)-\mathcal{N}_{Q_0}^+(r)\right)\\
&\quad\quad\quad-\left(\mathcal{N}'(Q_0+\tilde{r})-\mathcal{N}'(Q_0)-\mathcal{N}_{Q_0}^+(\tilde{r})\right)\Big\|_{H_{P_\epsilon}^{-1}}\\
&\leq \frac{1}{2}\|r-\tilde{r}\|_{H_{P_\epsilon}^1}.
\end{align*}
Therefore, we conclude that $\Phi$ is contractive, and it has a unique fixed point, denoted by $r_\epsilon$, on the ball of radius $\delta_\epsilon$ centered at $0$ in $H_{P_\epsilon}^1$. As a consequence, $u_\epsilon=Q_0+r_\epsilon$ solves the higher-order equation, and $\|u_\epsilon-Q_0\|_{H_{P_\epsilon}^1}=\|r_\epsilon\|_{H_{P_\epsilon}^1}\to 0$ as $\epsilon\to 0$. 
\end{proof}

\section{Local uniqueness}\label{sec: uniqueness}

The solution $u_\epsilon$, given by Proposition \ref{contraction}, is unique in a small ball of radially symmetric real-valued functions whose radius may depend on $\epsilon>0$. In this section, we upgrade this uniqueness to that in a small ball of all complex-valued functions whose radius is independent of $\epsilon>0$.

\begin{prop}[Uniqueness]\label{uniqueness}
Suppose that \eqref{uniform lower bound} (as well as \eqref{subcritical} for NLS) holds. Then, there exist $\delta>0$ and $\epsilon_0>0$ such that if $0<\epsilon\leq\epsilon_0$, then the solution $u_\epsilon$ to the higher-order equation \eqref{hNLS} (resp., \eqref{hNLH}), constructed in Proposition \ref{contraction}, is unique in a $\delta$-ball centered at $u_0$ in $H_{P_\epsilon}^1(\mathbb{R}^d;\mathbb{C})$ up to translation and phase shift.
\end{prop}

\begin{proof}
We prove the proposition only for the higher-order NLH, because the proof for the higher-order NLS follows similarly. Let $\epsilon_0>0$ be a small number given in Proposition \ref{prop: non-degeneracy estimates}, and let $\delta>0$ be sufficiently small numbers to be chosen later. Suppose that $0<\epsilon\leq\epsilon_0$ and $\tilde{u}_\epsilon$ is another solution to the higher-order equation in a $\delta$-ball centered at $u_0$ in $H_{P_\epsilon}^1(\mathbb{R}^3;\mathbb{C})$. 

First, we aim to show that the imaginary part of $\tilde{u}_\epsilon$ is orthogonal to $u_\epsilon$ up to phase shift. To this end, we consider 
$$\tilde{U}_\epsilon=\frac{\overline{\langle \tilde{u}_\epsilon, u_\epsilon\rangle_{H_{P_\epsilon}^1}}}{|\langle \tilde{u}_\epsilon, u_\epsilon\rangle_{H_{P_\epsilon}^1}|}\tilde{u}_\epsilon,$$
which also solves the higher-order equation. Here, since $u_\epsilon$ and $\tilde{u}_\epsilon$ are assumed to be sufficiently close to $u_0$, the denominator $\langle \tilde{u}_\epsilon, u_\epsilon\rangle_{H_{P_\epsilon}^1}\neq0$. Note that $\tilde{u}_\epsilon\mapsto\tilde{U}_\epsilon$ is a natural action, because if $\tilde{u}_\epsilon$ is simply a rotated $u_\epsilon$ on the complex plane, then this action rotates it back to $u_\epsilon$. Moreover, we have
\begin{align*}
\left\|\tilde{U}_\epsilon-u_\epsilon\right\|_{H_{P_\epsilon}^1}^2&=\|\tilde{U}_\epsilon\|_{H_{P_\epsilon}^1}^2+\|u_\epsilon\|_{H_{P_\epsilon}^1}^2-2\textup{Re}\langle \tilde{U}_\epsilon, u_\epsilon\rangle_{H_{P_\epsilon}^1}\\
&=\|\tilde{u}_\epsilon\|_{H_{P_\epsilon}^1}^2+\|u_\epsilon\|_{H_{P_\epsilon}^1}^2-2|\langle \tilde{u}_\epsilon, u_\epsilon\rangle_{H_{P_\epsilon}^1}|\\
&\leq\left|\|\tilde{u}_\epsilon\|_{H_{P_\epsilon}^1}^2+\|u_\epsilon\|_{H_{P_\epsilon}^1}^2-2\textup{Re}\langle \tilde{u}_\epsilon, u_\epsilon\rangle_{H_{P_\epsilon}^1}\right|=\|\tilde{u}_\epsilon-u_\epsilon\|_{H_{P_\epsilon}^1}^2\\
&\leq\left\{\|\tilde{u}_\epsilon-u_0\|_{H_{P_\epsilon}^1}+\|u_\epsilon-u_0\|_{H_{P_\epsilon}^1}\right\}^2\leq4\delta^2
\end{align*}
and
$$\langle \textup{Im}(\tilde{U}_\epsilon), u_\epsilon\rangle_{H_{P_\epsilon}^1}=\textup{Im}\left\{\frac{\overline{\langle \tilde{u}_\epsilon, u_\epsilon\rangle_{H_{P_\epsilon}^1}}}{|\langle \tilde{u}_\epsilon, u_\epsilon\rangle_{H_{P_\epsilon}^1}|}\langle \tilde{u}_\epsilon, u_\epsilon\rangle_{H_{P_\epsilon}^1}\right\}=\textup{Im}\Big\{|\langle \tilde{u}_\epsilon, u_\epsilon\rangle_{H_{P_\epsilon}^1}|\Big\}=0.$$
Therefore, replacing $\tilde{u}_\epsilon$ by $\tilde{U}_\epsilon$ and $\delta$ by $\frac{\delta}{2}$, we may assume that the imaginary part of $\tilde{u}_\epsilon$ is orthogonal to $u_\epsilon$ in $H_{P_\epsilon}^1$.

We denote the difference between two solutions by
$$r_\epsilon:=\tilde{u}_\epsilon-u_\epsilon=v_\epsilon+iw_\epsilon\quad\left(\Leftrightarrow\tilde{u}_\epsilon=(u_\epsilon+v_\epsilon)+iw_\epsilon\right),$$
where $v_\epsilon$ and $w_\epsilon$ are real-valued, and $\langle w_\epsilon, u_\epsilon\rangle_{H_{P_\epsilon}^1}=0$. When $u_\epsilon$ and $\tilde{u}_\epsilon$ are solutions to the higher-order NLH, then the difference $r_\epsilon$ satisfies 
\begin{align*}
(P_\epsilon+1) r_\epsilon&=\left(|x|^{-1}*|\tilde{u}_\epsilon|^2\right)\tilde{u}_\epsilon-\left(|x|^{-1}*|u_\epsilon|^2\right)u_\epsilon\\
&=\Big(|x|^{-1}*\left(u_\epsilon^2+2u_\epsilon v_\epsilon+|r_\epsilon|^2\right)\Big)\left((u_\epsilon+v_\epsilon)+iw_\epsilon\right)-\left(|x|^{-1}*|u_\epsilon|^2\right)u_\epsilon\\
&=\Big(|x|^{-1}*\left(2u_\epsilon v_\epsilon+|r_\epsilon|^2\right)\Big)u_\epsilon+\Big(|x|^{-1}*\left(u_\epsilon^2+2u_\epsilon v_\epsilon+|r_\epsilon|^2\right)\Big)(v_\epsilon+iw_\epsilon).
\end{align*}
Moving the linear terms on the right hand side to the left then using \eqref{LA relation}, the imaginary part of the equation (for $w_\epsilon$) can be written as
$$\mathcal{L}_\epsilon^-w_\epsilon=\Big(|x|^{-1}*\left(2u_\epsilon v_\epsilon+|r_\epsilon|^2\right)\Big)w_\epsilon.$$
Then, by Proposition \ref{prop: non-degeneracy estimates} and the nonlinear estimate (Lemma \ref{multilinear estimate}), we prove that
\begin{align*}
\beta\|w_\epsilon\|_{H_{P_\epsilon}^1}&\leq\|\mathcal{L}_\epsilon^-w_\epsilon\|_{H_{P_\epsilon}^{-1}}=\left\|\Big(|x|^{-1}*\left(2u_\epsilon v_\epsilon+|r_\epsilon|^2\right)\Big)w_\epsilon\right\|_{H_{P_\epsilon}^{-1}}\\
&\leq C \left(\|u_\epsilon\|_{H_{P_\epsilon}^1}+\|r_\epsilon\|_{H_{P_\epsilon}^1}\right)\|r_\epsilon\|_{H_{P_\epsilon}^1}\|w_\epsilon\|_{H_{P_\epsilon}^1}\\
&\leq C\left(\|u_0\|_{H_{P_\epsilon}^1}+2\delta\right)\delta\|w_\epsilon\|_{H_{P_\epsilon}^1}.
\end{align*}
Therefore, choosing small $\delta$ such that $C(\|u_0\|_{H_{P_\epsilon}^1}+2\delta)\delta<\beta$, we conclude that $w_\epsilon=0$.

By a suitable phase shift, we may assume that $\tilde{u}_\epsilon$ is real-valued. Furthermore, by translating $\tilde{u}_\epsilon$ so that $\|\tilde{u}_\epsilon(\cdot-a)-u_\epsilon\|_{H_{P_\epsilon}^1}=\|\tilde{u}_\epsilon-u_\epsilon(\cdot+a)\|_{H_{P_\epsilon}^1}$ is minimized, equivalently
$$\frac{\partial}{\partial_{x_j}}\Bigg|_{a=0}\|\tilde{u}_\epsilon-u_\epsilon(\cdot+a)\|_{H_{P_\epsilon}^1}^2=0,$$
we may assume that $\tilde{u}_\epsilon$ is orthogonal to $\partial_{x_j}u_\epsilon$ in $H_{P_\epsilon}^1$ for all $j=1,2,3$. Now, we write the equation for the difference $r_\epsilon=\tilde{u}_\epsilon-u_\epsilon$,
\begin{align*}
(P_\epsilon+1) r_\epsilon&=\mathcal{N}'(\tilde{u}_\epsilon)-\mathcal{N}'(u_\epsilon)=\mathcal{N}'(u_\epsilon+r_\epsilon)-\mathcal{N}'(u_\epsilon)\\
\Rightarrow\mathcal{L}_\epsilon^+r_\epsilon&=\mathcal{N}'(u_\epsilon+r_\epsilon)-\mathcal{N}'(u_\epsilon)-\mathcal{N}_{u_\epsilon}^+(r_\epsilon),
\end{align*}
where $\mathcal{N}_{u_\epsilon}^+$ is defined in \eqref{N+}. Since $r_\epsilon=\tilde{u}_\epsilon-u_\epsilon$ is orthogonal to $\partial_{x_j}u_\epsilon$ in $H_{P_\epsilon}^1$, by Proposition \ref{prop: non-degeneracy estimates} and the nonlinear estimate (Lemma \ref{multilinear estimate}) again, we obtain
\begin{align*}
\beta\|r_\epsilon\|_{H_{P_\epsilon}^1}&\leq\|\mathcal{L}_\epsilon^+r_\epsilon\|_{H_{P_\epsilon}^{-1}}=\left\|\left(|x|^{-1}*r_\epsilon^2\right)u_\epsilon+\Big(|x|^{-1}*\left(2u_\epsilon r_\epsilon+r_\epsilon^2\right)\Big)r_\epsilon\right\|_{H_{P_\epsilon}^{-1}}\\
&\leq \tilde{C} \left(\|u_\epsilon\|_{H_{P_\epsilon}^1}+\|r_\epsilon\|_{H_{P_\epsilon}^1}\right)\|r_\epsilon\|_{H_{P_\epsilon}^1}^2\\
&\leq \tilde{C} \left(\|u_0\|_{H_{P_\epsilon}^1}+2\delta\right)\delta\|r_\epsilon\|_{H_{P_\epsilon}^1}.
\end{align*}
Choosing even smaller $\delta>0$ such that $\tilde{C}(\|u_0\|_{H_{P_\epsilon}^1}+2\delta)\delta<\beta$ if necessary, we prove that $r_\epsilon=0$. Therefore, we conclude that $\tilde{u}_\epsilon=u_\epsilon$ up to translation and phase shift.
\end{proof}

Now, we are ready to prove the main theorem.
\begin{proof}[Proof of Theorem \ref{main thm}]
By Proposition \ref{convergence}, if $\epsilon>0$ is small enough, then ground states $Q_\epsilon^J$'s are close to the reference ground state $Q_0$ in $H_{P_\epsilon}^1$ (modifying the sequence by translation and phase shift if necessary). However, by uniqueness in Proposition \ref{uniqueness}, $Q_\epsilon^J$ is identified with the radially symmetric real-valued solution $u_\epsilon$, constructed in Proposition \ref{contraction}. Moreover, by Proposition \ref{Non-degeneracy epsilon}, it is non-degenerate. Therefore, we  prove the main theorem.
\end{proof}

\section{Proof of Theorem \ref{approximation}}

We proceed exactly as in the proof of Proposition \ref{contraction}. We denote by
$$r_c^J:=Q_c^J-Q_c$$
the difference between two solutions. Then, it satisfies 
\begin{align*}
(P_c^J+1) r_c^J&=(P_c^J+1)Q_c^J-(P_c^J+1)Q_c\\
&=\left(P_c-P_c^J\right)Q_c+(P_c^J+1)Q_c^J-(P_c+1)Q_c\\
&=\left(P_c-P_c^J\right)Q_c+\mathcal{N}'({Q_c^J})-\mathcal{N}'({Q_c})\\
&=\left(P_c-P_c^J\right)Q_c+\mathcal{N}'({Q_c+r_c})'-\mathcal{N}'({Q_c}),
\end{align*}
where $P_c=\sum_{j=1}^J\frac{(-1)^{j-1}\alpha_j}{m^{2j-1}c^{2j-2}}$ and $\alpha_j=\frac{(2j-2)!}{j!(j-1)!2^{2j-1}}$. Moving the linear terms on the right hand side to the left as above, we write
\begin{equation}\label{reformulated higher-order eq}
\mathcal{L}_c^{J,+} r_c^J=\left(P_c-P_c^J\right)Q_c+\mathcal{N}'({Q_c+r_c})'-\mathcal{N}'({Q_c})-\mathcal{N}_{Q_c}^+(r_c),
\end{equation}
where
$$\mathcal{L}_c^{J,+}=P_c^J+1-\mathcal{N}_{Q_c}^+=\sqrt{1+P_c^J}\left(\textup{Id}-\frac{1}{\sqrt{1+P_c^J}}\mathcal{N}_{Q_c}^+\frac{1}{\sqrt{1+P_c^J}}\right)\sqrt{1+P_c^J}.$$
Repeating the proof of Lemma \ref{radial lower bound} together with $Q_c\to Q_0$ in $H^1$ as $c\to \infty$, one can show that 
\begin{align*}
&\textup{Id}-\frac{1}{\sqrt{1+P_c^J}}\mathcal{N}_{Q_c}^+\frac{1}{\sqrt{1+P_c^J}}\\
&=\left(\textup{Id}-\frac{1}{\sqrt{1-\Delta}}\mathcal{N}_{Q_0}^+\frac{1}{\sqrt{1-\Delta}}\right)+\left(\frac{1}{\sqrt{1-\Delta}}\mathcal{N}_{Q_0}^+\frac{1}{\sqrt{1-\Delta}}-\frac{1}{\sqrt{1+P_c^J}}\mathcal{N}_{Q_c}^+\frac{1}{\sqrt{1+P_c^J}}\right)
\end{align*}
is invertible on $L_{rad}^2(\mathbb{R}^d;\mathbb{R})$ and its inverse is uniformly bounded for sufficiently large $c\geq1$. Hence, applying the trivial embedding $H_{P_c^J}^1\hookrightarrow H^1$ (from Lemma \ref{uniform lower bound of hpNLH}) and its dual embedding, we obtain 
\begin{align*}
\|r_c^J\|_{H^1}&\lesssim \|r_c^J\|_{H_{P_c^J}^1}=\left\|(\mathcal{L}_c^{J,+})^{-1}\Bigg\{\left(P_c-P_c^J\right)Q_c+\mathcal{N}'({Q_c+r_c})'-\mathcal{N}'({Q_c})-\mathcal{N}_{Q_c}^+(r_c)\Bigg\}\right\|_{H_{P_c^J}^1}\\
&\lesssim\left\|\left(P_c-P_c^J\right)Q_c+\mathcal{N}'({Q_c+r_c})'-\mathcal{N}'({Q_c})-\mathcal{N}_{Q_c}^+(r_c)\right\|_{H_{P_c^J}^{-1}}\\
&\lesssim\left\|\left(P_c-P_c^J\right)Q_c\right\|_{H^{-1}}+\left\|\mathcal{N}'({Q_c+r_c})'-\mathcal{N}'({Q_c})-\mathcal{N}_{Q_c}^+(r_c)\right\|_{H^{-1}},
\end{align*}
where the implicit constants do not depend on $c\geq1$. Therefore, it follows from the nonlinear estimates (Lemma \ref{multilinear estimate}) that for sufficiently large $c\geq1$, 
$$\|r_c^J\|_{H^1}\lesssim\left\|\left(P_c-P_c^J\right)Q_c\right\|_{H^{-1}}\lesssim \frac{1}{c^{2J}}\|Q_c\|_{H^{2J+1}},$$
because by Taylor's theorem, 
$$\left|\left(\sqrt{c^2s+m^2c^4}-mc^2\right)-\sum_{j=1}^J\frac{(-1)^{j-1}\alpha_j}{m^{2j-1}c^{2j-2}}s^j\right|\lesssim\frac{s^{J+1}}{c^{2J}}.$$
Finally, by the uniform bound on high Sobolev norm of $Q_c$ (see Proposition \ref{upgraded convergence} or \cite{CHS1}), we conclude that $\|Q_c^J-Q_c\|_{H^1}=\|r_c^J\|_{H^1}\lesssim_J c^{-2J}$.

\appendix

\section{Nonlinear estimates}

We show the nonlinear estimates which are used in the contraction mapping argument. 

\begin{lem}[Nonlinear estimates]\label{nonlinear estimates}
Let $u\in H^1$ be real-valued. For any $\eta>0$, there exists $\delta_0>0$, depending on $\|u\|_{H^1(\mathbb{R}^d;\mathbb{R})}$ and $\eta$, such that if $0<\delta\leq\delta_0$ and
$$\|r\|_{H^1(\mathbb{R}^d;\mathbb{R})}, \|\tilde{r}\|_{H^1(\mathbb{R}^d;\mathbb{R})}\leq \delta,$$
then
$$\left\|\mathcal{N}'(u+r)-\mathcal{N}'(u)-\mathcal{N}_{u}^+(r)\right\|_{L^2(\mathbb{R}^d;\mathbb{R})}\leq \eta\|r\|_{H^1(\mathbb{R}^d;\mathbb{R})}$$
and
$$\left\|\left(\mathcal{N}'(u+r)-\mathcal{N}'(u)-\mathcal{N}_{u}^+(r)\right)-\left(\mathcal{N}'(u+\tilde{r})-\mathcal{N}'(u)-\mathcal{N}_{u}^+(\tilde{r})\right)\right\|_{L^2(\mathbb{R}^d;\mathbb{R})}\leq\eta\|r-\tilde{r}\|_{H^1(\mathbb{R}^d;\mathbb{R})}.$$
\end{lem}

The above lemma follows from the multilinear estimates.

\begin{lem}[Multilinear estimates]\label{multilinear estimate}
We have
$$\left\|\left(\frac{1}{|x|}*(\phi_1\phi_2)\right)\phi_3\right\|_{L^2(\mathbb{R}^3;\mathbb{R})}\lesssim \prod_{j=1}^3\|\phi_j\|_{H^1(\mathbb{R}^3;\mathbb{R})}.$$
Moreover, if $d=1,2$ and $k\in\mathbb{N}$ or if $d=3$ and $k=1$, then
$$\left\|\prod_{j=1}^{2k+1}\phi_j\right\|_{L^2(\mathbb{R}^d;\mathbb{R})}\lesssim \prod_{j=1}^{2k+1}\|\phi_j\|_{H^1(\mathbb{R}^d;\mathbb{R})}.$$
\end{lem}

\begin{proof}
By the H\"older, Young's and Sobolev inequalities, we prove that
\begin{align*}
\left\|\left(\frac{1}{|x|}*(\phi_1\phi_2)\right)\phi_3\right\|_{L^2(\mathbb{R}^3;\mathbb{R})}&\leq \left\|\left(\frac{1}{|x|}*(\phi_1\phi_2)\right)\right\|_{L^9(\mathbb{R}^3;\mathbb{R})}\|\phi_3\|_{L^{18/7}(\mathbb{R}^3;\mathbb{R})}\\
&\lesssim\|\phi_1\phi_2\|_{L^{9/7}(\mathbb{R}^3;\mathbb{R})}\|\phi_3\|_{L^{18/7}(\mathbb{R}^3;\mathbb{R})}\\
&\lesssim\prod_{j=1}^3\|\phi_j\|_{L^{18/7}(\mathbb{R}^3;\mathbb{R})}\lesssim\prod_{j=1}^3\|\phi_j\|_{H^1(\mathbb{R}^3;\mathbb{R})}
\end{align*}
and similarly,
$$\left\|\prod_{j=1}^{2k+1}\phi_j\right\|_{L^2(\mathbb{R}^d;\mathbb{R})}\leq \prod_{j=1}^{2k+1}\|\phi_j\|_{L^{2(2k+1)}(\mathbb{R}^d;\mathbb{R})}\lesssim \prod_{j=1}^{2k+1}\|\phi_j\|_{H^1(\mathbb{R}^d;\mathbb{R})}.$$
\end{proof}

\begin{proof}[Proof of Lemma \ref{nonlinear estimates}]
Suppose that $\|r\|_{H^1},\|\tilde{r}\|_{H^1}\leq\|u\|_{H^1}$. For the Hartree nonlinearity, by algebra, we write
$$\mathcal{N}'(u+r)-\mathcal{N}'(u)-\mathcal{N}_{u}^+(r)=\left(\frac{1}{|x|}*r^2\right)u+2\left(\frac{1}{|x|}*(ur)\right)r+\left(\frac{1}{|x|}*r^2\right)r$$
and
\begin{align*}
&\left(\mathcal{N}'(u+r)-\mathcal{N}'(u)-\mathcal{N}_{u}^+(r)\right)-\left(\mathcal{N}'(u+\tilde{r})-\mathcal{N}'(u)-\mathcal{N}_{u}^+(\tilde{r})\right)\\
&=\left(\frac{1}{|x|}*\left((r+\tilde{r})(r-\tilde{r})\right)\right)u+2\left(\frac{1}{|x|}*\left(u(r-\tilde{r})\right)\right)r+2\left(\frac{1}{|x|}*(u\tilde{r})\right)(r-\tilde{r})\\
&\quad+\left(\frac{1}{|x|}*\left((r+\tilde{r})(r-\tilde{r})\right)\right)r+\left(\frac{1}{|x|}*\tilde{r}^2\right)(r-\tilde{r})
\end{align*} 
Thus, by the multilinear estimate (Lemma \ref{multilinear estimate}), 
$$\left\|\mathcal{N}'(u+r)-\mathcal{N}'(u)-\mathcal{N}_{u}^+(r)\right\|_{L^2}\leq C\Big(\|u\|_{H^1}+\|r\|_{H^1}\Big)\|r\|_{H^1}^2\leq 2C\delta\|u\|_{H^1}\|r\|_{H^1}$$
and
\begin{align*}
&\left\|\left(\mathcal{N}'(u+r)-\mathcal{N}'(u)-\mathcal{N}_{u}^+(r)\right)-\left(\mathcal{N}'(u+\tilde{r})-\mathcal{N}'(u)-\mathcal{N}_{u}^+(\tilde{r})\right)\right\|_{L^2}\\
&\leq C\Big(\|u\|_{H^1}+\|r\|_{H^1}+\|\tilde{r}\|_{H^1}\Big)\Big(\|r\|_{H^1}+\|\tilde{r}\|_{H^1}\Big)\|r-\tilde{r}\|_{H^1}\\
&\leq 6C\delta\|u\|_{H^1}\|r-\tilde{r}\|_{H^1}.
\end{align*}
Then, taking $\delta_0=\eta\min\{\frac{1}{6C\|u\|_{H^1}},\|u\|_{H^1}\}$, we prove the lemma for the Hartree nonlinearity.

Similarly, for the polynomial nonlinearity, by the multilinear estimate (Lemma \ref{multilinear estimate}), 
\begin{align*}
\left\|\mathcal{N}'(u+r)-\mathcal{N}'(u)-\mathcal{N}_{u}^+(r)\right\|_{L^2}&=\left\|\sum_{j=2}^{2k+1} \begin{pmatrix}
    2k+1  \\
    j \\
\end{pmatrix} u^{2k+1-j} r^j\right\|_{L^2}\\
&\leq \sum_{j=2}^{2k+1} \begin{pmatrix}
    2k+1  \\
    j \\
\end{pmatrix} \left\|u^{2k+1-j} r^j\right\|_{L^2}\\
&\leq C\sum_{j=2}^{2k+1} \begin{pmatrix}
    2k+1  \\
    j \\
\end{pmatrix} \|u\|_{H^1}^{2k+1-j} \|r\|_{H^1}^j\\
&\leq C_k\delta\|u\|_{H^1}^{2k-1}\|r\|_{H^1}
\end{align*}
and
\begin{align*}
&\left\|\left(\mathcal{N}'(u+r)-\mathcal{N}'(u)-\mathcal{N}_{u}^+(r)\right)-\left(\mathcal{N}'(u+\tilde{r})-\mathcal{N}'(u)-\mathcal{N}_{u}^+(\tilde{r})\right)\right\|_{L^2}\\
&=\left\|\sum_{j=3}^{2k+1} \begin{pmatrix}
    2k+1  \\
    j \\
\end{pmatrix} u^{2k+1-j} (r-\tilde{r})(r^{j-1}+r^{j-2}\tilde{r}+\cdots+\tilde{r}^{j-1})\right\|_{L^2}\\
&\leq\sum_{j=3}^{2k+1} \begin{pmatrix}
    2k+1  \\
    j \\
\end{pmatrix} \left\|u^{2k+1-j} (r-\tilde{r})(r^{j-1}+r^{j-2}\tilde{r}+\cdots+\tilde{r}^{j-1})\right\|_{L^2}\\
&\leq C\sum_{j=3}^{2k+1} \begin{pmatrix}
    2k+1  \\
    j \\
\end{pmatrix} \|u\|_{H^1}^{2k+1-j} \|r-\tilde{r}\|_{H^1}\left(\|r\|_{H^1}^{j-1}+\|r\|_{H^1}^{j-2}\|\tilde{r}\|_{H^1}+\cdots+\|\tilde{r}\|_{H^1}^{j-1}\right)\\
&\leq C_k\delta\|u\|_{H^1}^{2k-1}\|r-\tilde{r}\|_{H^1}
\end{align*}
for some constant $C_k>0$. Then, taking $\delta_0=\eta\min\{\frac{1}{2C_k\|u\|_{H^1}^{2k-1}},\|u\|_{H^1}\}$, we complete the proof of the lemma for the polynomial nonlinearity.
\end{proof}

\section{Uniform lower bound for higher-order operators in \eqref{hpNLH}}

\begin{lem}[Uniform lower bound for higher-order operators in \eqref{hpNLH}]\label{uniform lower bound of hpNLH}
\label{lower-bound}
For any $\xi\in\mathbb{R}^3$, we have
$$\sum_{j=1}^{2k-1}\frac{(-1)^{j-1}\alpha_j}{m^{2j-1}c^{2j-2}}|\xi|^{2j}\geq \frac{|\xi|^2}{2m},$$
where $a_j=\frac{(2j-2)!}{j!(j-1)! 2^{2j-1}}$.
\end{lem}

\begin{proof}
By change of variables $\frac{\xi}{m}\mapsto\xi$, it suffices to prove the lemma assuming $m=1$. The inequality is trivial when $k=1$. Suppose that $k\geq 2$.  Splitting the positive and the negative terms and then applying the Cauchy-Schwarz inequality for the negative terms, we obtain
\begin{align*}
\sum_{j=1}^{2k-1}\frac{(-1)^{j-1}\alpha_j}{c^{2j-2}}|\xi|^{2j}&=\sum_{j=1}^k\frac{\alpha_{2j-1}}{c^{4j-4}}|\xi|^{4j-2}-\sum_{j=1}^{k-1}\frac{\alpha_{2j}}{c^{4j-2}}|\xi|^{4j}\\
&\geq \sum_{j=1}^k\frac{\alpha_{2j-1}}{c^{4j-4}}|\xi|^{4j-2}-\frac{1}{2}\sum_{j=1}^{k-1} \left\{\frac{(\alpha_{2j})^2}{\alpha_{2j-1}\alpha_{2j+1}}\frac{\alpha_{2j-1}}{c^{4j-4}}|\xi|^{4j-2}+\frac{\alpha_{2j+1}}{c^{4j}}|\xi|^{4j+2}\right\}.
\end{align*}
Since $\frac{(\alpha_{2j})^2}{\alpha_{2j-1}\alpha_{2j+1}}=\cdots=\frac{(4j-3)(2j+1)}{(4j-1)2j}\leq 1$ for all $j\geq 1$, it is bounded below from
\begin{align*}
&\sum_{j=1}^k\frac{\alpha_{2j-1}}{c^{4j-4}}|\xi|^{4j-2}-\frac{1}{2}\sum_{j=1}^{k-1} \left\{\frac{\alpha_{2j-1}}{c^{4j-4}}|\xi|^{4j-2}+\frac{\alpha_{2j+1}}{c^{4j}}|\xi|^{4j+2}\right\}\\
&= \sum_{j=1}^k\frac{\alpha_{2j-1}}{c^{4j-4}}|\xi|^{4j-2}-\frac{1}{2}\sum_{j=1}^{k-1}\frac{\alpha_{2j-1}}{c^{4j-4}}|\xi|^{4j-2}-\frac{1}{2}\sum_{j=2}^k\frac{\alpha_{2j-1}}{c^{4j-4}}|\xi|^{4j-2}\\
&=\frac{1}{2}|\xi|^2+\frac{\alpha_{2k-1}}{2c^{4k-4}}|\xi|^{4k-2}\geq\frac{1}{2}|\xi|^2.
\end{align*}
\end{proof}

\end{document}